\documentclass[12pt,english,reqno]{amsart}  

\usepackage[T1]{fontenc} 
\usepackage{geometry} 
\usepackage{amssymb,amsmath,amsthm} 

\usepackage{pdfpages}
\usepackage{graphicx}
\usepackage{tikz} 
\usetikzlibrary{arrows,automata} 

\usepackage{hyperref}
\hypersetup{
    colorlinks=true,
    linkcolor=blue,
    filecolor=magenta,      
    urlcolor=cyan,
}

\usepackage[utf8]{inputenc}

\newtheorem{theorem}{Theorem}[section]

\newtheorem{lemma}[theorem]{Lemma}

\newtheorem{remark}[theorem]{Remark}

\newcommand{\bfS}[1]{\mathbf{S}^{#1}} 
\newcommand{\bfR}[1]{\mathbf{R}^{#1}} 

\title{The stability index and Yau's conjecture for Carlotto-Schulz minimal hypertori}
\author{Oscar Perdomo}
\address{Central Connecticut State University}
\email{perdomoosm@ccsu.edu}
\date{\today}

\begin{document}

\maketitle

\begin{abstract}
Recently, for any \(n>1\), Carlotto and Schulz showed the existence of a minimal embedding \(X_{CS}^{n}:\bfS{n-1}\times\bfS{n-1}\times\bfS{1}\to \bfS{2n}\). In this paper, we show that the stability index of these embedded minimal hypersurfaces is at least \(n^2+4n+3\). We also show that Yau's conjecture holds for these examples if and only if the solution of the differential equation \(z''(t)+a_n(t)\,z'(t)+(2n-1)z(t)=0\) with \(z(0)=1\) and \(z'(0)=0\) satisfies \(z'(T)>0\). Here, \(T\) and the \(T\)-periodic function \(a_n(t)\) are determined in terms of the functions defining the immersion \(X_{CS}^n\).
\end{abstract}

\section{Introduction}

Let \(M \subset \bfS{N+1} \subset \bfR{N+2}\) be an \(N\)-dimensional closed minimal hypersurface.  
If \(\nu : M \to \bfR{N+2}\) denotes its Gauss map and \(\bar{\nabla}\) the Levi-Civita connection on \(\bfS{N+1}\), then for any \(m \in M\) and \(v \in T_m M\), the shape operator \(A\) is given by  
\[
A_m(v) = -\bar{\nabla}_v \nu.
\]
Since \(M\) is minimal, the eigenvalues of \(A\), \(\kappa_1, \dots, \kappa_N\), known as the principal curvatures, satisfy \(\kappa_1 + \dots + \kappa_N = 0\).  
The stability operator on \(M\) is
\[
J = -\Delta - |A|^2 - N,
\]
where \(|A|^2 = \kappa_1^2 + \dots + \kappa_N^2\).  
The {\it stability index} of \(M\) is the number of negative eigenvalues (counted with multiplicity) of the stability operator.

Since \(M \subset \bfR{N+2}\), for \(i = 1, \dots, N+2\) we can naturally define the functions
\[
x_i : M \to \mathbf{R}, \quad \nu_i : M \to \mathbf{R},
\]
where \((x_1, \dots, x_{N+2})\) gives the immersion of \(M\) and \((\nu_1, \dots, \nu_{N+2})\) are the coordinate functions of the Gauss map.  
It is known that
\begin{eqnarray}\label{delta-J}
-\Delta x_i = N x_i, \quad J(\nu_i) = -N \nu_i, \quad \text{and} \quad J(x_i \nu_j - x_j \nu_i) = 0.
\end{eqnarray}

When \(M\) is totally geodesic, that is,
\[
M = \bfS{N} = \{ x = (x_1, \dots, x_{N+2}) \in \bfR{N+2} : x_{N+2} = 0 \},
\]
all principal curvatures vanish and the only negative eigenvalue of \(J\) is \(-N\) with multiplicity one.  
Therefore, the stability index of \(M\) is \(1\).

When \(M\) is not totally geodesic, the functions \(\nu_1, \dots, \nu_{N+2}\) are linearly independent. Since \(J(\nu_i) = -N \nu_i\), we see that \(-N\) is an eigenvalue of \(J\) with multiplicity at least \(N+2\).  
As the first eigenvalue of \(J\) must have multiplicity one, it follows that the stability index is at least \(N+3\).

One of the few situations where the spectrum of the Laplacian is completely known is when \(M\) is a sphere (Theorem~\ref{SpectrumSphere}).  
Using this theorem, it can be checked that the stability index of the Clifford hypersurfaces
\[
\bfS{k}\!\left( \sqrt{\frac{k}{N}} \right) \times \bfS{N-k}\!\left( \sqrt{\frac{N-k}{N}} \right)
\]
is exactly \(N+3\).  
A well-known conjecture states that the Clifford hypersurfaces are the only minimal hypersurfaces with index \(N+3\).  
This conjecture has been proven only in the case \(N=2\) by Urbano \cite{U}.

When the hypersurface has antipodal symmetry, the author in \cite{P2001} showed that the index must be greater than \(N+3\), by proving the existence of an \((N+4)^{\text{th}}\) eigenvalue of \(J\) between its first eigenvalue and \(-N\).   
Interesting connections between the stability index and topology can be found in \cite{S2010} and \cite{CSW}.

A related conjecture to the one on the stability index is {\it Yau's conjecture} \cite{Yau-Seminar}, which states that for any embedded minimal hypersurface \(M \subset \bfS{N+1}\), the first nonzero eigenvalue of the Laplacian is \(N\).  
The condition that \(M\) is embedded is crucial: the conjecture should be stated as {\it \(M\) is embedded and minimal if and only if \(N\) is the first nonzero eigenvalue of the Laplacian}, since all known non-embedded examples have first nonzero eigenvalue smaller than \(N\).  
In 1983, Choi and Wang \cite{Choi-Wang} showed that if \(M\) is embedded, then the first nonzero eigenvalue is greater than or equal to \(N/2\).  
The conjecture remains open even in the case of surfaces.  
The best result for surfaces is given in \cite{Choe-Soret}, where the authors show that the conjecture holds for surfaces invariant under a special group of symmetries.

One of the main challenges in Yau's conjecture is that, in general, computing the spectrum of the Laplacian of a manifold is very difficult.  
The complete spectrum is fully known only for spheres, products of spheres (Clifford hypersurfaces), and the cubic isoparametric examples studied by Solomon \cite{Solomon1,Solomon2}.  
In the latter case, Solomon was able to verify Yau's conjecture.  
The verification of Yau's conjecture for all isoparametric hypersurfaces was later obtained by Tang and Yan \cite{Tang-Yan}. 

Recently, for any integer \(n>1\), Carlotto and Schulz \cite{CS} constructed an embedded minimal hypersurface 
\[
X^n_{CS}:\bfS{n-1}\times \bfS{n-1}\times \bfS{1} \to \bfS{2n}.
\]
In this paper we show that, for any \(n>1\), the stability index of \(X^n_{CS}\) is greater than or equal to \(n^2+4n+3\).  
For \(n=2\), Carlotto, Schulz, and Wiygul \cite{CSW} showed that the stability index of the hypertorus is at least \(8\); our result improves this lower bound to \(15\).  

For any \(n>1\), we also show that Yau's conjecture holds for \(X^n_{CS}\) if and only if the solution of the second-order differential equation  
\[
z'' + a_n(t) z' + (2n-1) z = 0, \qquad z(0) = 1, \quad z'(0) = 0,
\]
satisfies \(z'(T) > 0\), where the period \(T\) and the \(T\)-periodic function \(a_n(t)\) are given in terms of the functions defining the immersion \(X^n_{CS}\).

For minimal surfaces in $\bfS{3}$, Ros proved that every equator cuts the surface into two components—the {\it two-piece property} \cite{Ros-TwoPiece}. Equivalently, each coordinate function has exactly two nodal domains, as would follow if Yau's conjecture held in that setting (since $-\Delta x_i=2\,x_i$ on a minimal surface in $\bfS{3}$). In the present context of the Carlotto–Schulz hypersurfaces, the verification of Yau's conjecture reduces to a one-dimensional question: deciding whether the known eigenvalue $2n-1$, together with a $T$-periodic eigenfunction that has exactly two zeros in $[0,T)$, is the first (positive) eigenvalue of the associated Hill operator on $\bfS{1}$.

\section{The Carlotto–Schulz embeddings}
For any integer $n>1$, consider the immersion $X:\bfS{n-1}\times \bfS{n-1}\times \bfS{1}\to \bfS{2n}\subset\bfR{2n+1}$ given by
\begin{eqnarray}\label{eq:Immersion}
X(y,z,t)=\bigl(\sin r(t)\,\cos\theta(t)\,y,\;\sin r(t)\,\sin\theta(t)\,z,\;\cos r(t)\bigr),
\end{eqnarray}
where $y,z\in\bfS{n-1}$. Under the assumption that the curve
\begin{eqnarray}\label{eq:ImmersionComponents}
\gamma(t)=\bigl(\gamma_1(t),\;\gamma_2(t),\;\gamma_3(t)\bigr)=\bigl(\sin r(t)\,\cos\theta(t),\;\sin r(t)\,\sin\theta(t),\;\cos r(t)\bigr)
\end{eqnarray}
is parametrized by arc length, there exists a function $\alpha(t)$ such that
\begin{equation}\label{eq:arclength}
\begin{aligned}
r'(t)&=\cos\alpha(t),\\
\theta'(t)&=\dfrac{\sin\alpha(t)}{\sin r(t)}.
\end{aligned}
\end{equation}
The curve $\gamma$ is called the \emph{profile curve}. A direct computation shows that the Gauss map of $X$ has the form
\[
\nu(t,y,z)=\bigl(\nu_1(t)\,y,\;\nu_2(t)\,z,\;\nu_3(t)\bigr),
\]
with
\begin{equation}\label{eq:GaussComponents}
\begin{aligned}
\nu_1(t)&=\cos r(t)\,\sin\alpha(t)\,\cos\theta(t)+\cos\alpha(t)\,\sin\theta(t),\\
\nu_2(t)&=\cos r(t)\,\sin\alpha(t)\,\sin\theta(t)-\cos\alpha(t)\,\cos\theta(t),\\
\nu_3(t)&=-\sin r(t)\,\sin\alpha(t).
\end{aligned}
\end{equation}
As shown in \cite{CS}, the immersion $X$ is minimal if and only if
\begin{equation}\label{eq:minimality}
\alpha'(t)=(2n-2)\,\csc r(t)\,\cos\alpha(t)\,\cot\bigl(2\theta(t)\bigr)
-(2n-1)\,\cot r(t)\,\sin\alpha(t).
\end{equation}
Substituting \eqref{eq:arclength} and \eqref{eq:minimality} yields the principal curvatures
\begin{align*}
\kappa_u(t)&=\csc r(t)\,\cos\alpha(t)\,\tan\theta(t)+\cot r(t)\,\sin\alpha(t),\\
\kappa_v(t)&=\cot r(t)\,\sin\alpha(t)-\csc r(t)\,\cos\alpha(t)\,\cot\theta(t),\\
\kappa_t(t)&=2\,\csc r(t)\,\cos\alpha(t)\,\cot\bigl(2\theta(t)\bigr)-2\,\cot r(t)\,\sin\alpha(t).
\end{align*}
Therefore, the square of the norm of the shape operator is
\[
|A|^2=(n-1)\,\kappa_u^2+(n-1)\,\kappa_v^2+\kappa_t^2.
\]
The existence of the embedding \(X\) follows from the following theorem in \cite{CS}:

\begin{theorem}[Existence of the  Carlotto–Schulz  embedding \cite{CS}]\label{ThmCS}
For any integer $n>1$, there exist \(s^* = T/4\) and \(r_0\in(0,\pi)\) such that the unique solution 
\(\bigl(r(t),\theta(t),\alpha(t)\bigr)\) of the system 
\eqref{eq:arclength}–\eqref{eq:minimality}, with initial conditions
\[
\theta(0)=\frac{\pi}{4},\quad
r(0)=r_0,\quad
\alpha(0)=-\frac{\pi}{2},
\]
satisfies at \(t=s^*\)
\[
\theta(s^*)>0,\quad
r(s^*)=\frac{\pi}{2},\quad
\alpha(s^*)=0.
\]
Moreover, on the interval \([0,s^*]\), the function \(\theta(t)\) is strictly decreasing, while \(r(t)\) and \(\alpha(t)\) are strictly increasing.
\end{theorem}

\begin{remark}\label{remark on solution}
Due to the symmetries of the system, the graph of \(\theta(t)\) is symmetric about the vertical lines \(t=\tfrac{T}{4}\) and \(t=\tfrac{3T}{4}\). It also has odd symmetry with respect to the points \(\bigl(\tfrac{T}{2},\tfrac{\pi}{4}\bigr)\) and \(\bigl(0,\tfrac{\pi}{4}\bigr)\). Likewise, the graph of \(r(t)\) is symmetric about the lines \(t=0\) and \(t=\tfrac{T}{2}\) and has odd symmetry with respect to the points \(\bigl(\tfrac{T}{4},\tfrac{\pi}{2}\bigr)\) and \(\bigl(\tfrac{3T}{4},\tfrac{\pi}{2}\bigr)\). Consequently, \(\theta(t)\) and \(r(t)\) are \(T\)-periodic. The function \(\alpha(t)\) is not periodic, but it satisfies \(\alpha(t+T)=\alpha(t)+2\pi\).
\end{remark}

\begin{remark}\label{translation}
From now on we translate the solution described in Theorem~\ref{ThmCS} by $T/4$ so that
$\theta(0)>0$ is the minimum of $\theta(t)$, $r(0)=\tfrac{\pi}{2}$, and $\alpha(0)=0$. 
Since the system is autonomous, its solutions are invariant under translations. Figure~\ref{fig:theta-r-alpha} shows the symmetries of $\theta$, $r$, and $\alpha$ after this translation.
\end{remark}

We end this section with a lemma that follows from the previous remark and will play an important role in the proof of our main results.

\begin{lemma}\label{simetrias}
Assume that $r(t)$, $\theta(t)$, and $\alpha(t)$ solve the ODE defining a minimal embedding in $\bfS{2n}$ and satisfy $0<\theta_0=\theta(0)<\tfrac{\pi}{4}$, where $\theta_0$ is the minimum of $\theta(t)$. If $s^*=T/4$ and $r_0$ are as in Theorem~\ref{ThmCS}, then:
\begin{itemize}
\item The functions \(\theta(t)\) and \(f_{\theta}(t)=\tfrac{\pi}{4} -\bigl(\tfrac{\pi}{4}-\theta_0\bigr)\cos\bigl(\tfrac{2\pi}{T}t\bigr)\) have the same symmetries and the same critical points. Moreover, \(\theta'(t)>0\) if and only if \(f_{\theta}'(t)>0\).
\item The functions \(r(t)\) and \(f_{r}(t)=\tfrac{\pi}{2} +\bigl(\tfrac{\pi}{2}-r_0\bigr)\sin\bigl(\tfrac{2\pi}{T}t\bigr)\) have the same symmetries and the same critical points. Moreover, \(r'(t)>0\) if and only if \(f_{r}'(t)>0\).
\item The functions \(\cos(\alpha(t))\) and \(\cos\bigl(\tfrac{2\pi}{T}t\bigr)\) have the same symmetries and the same critical points. Moreover, they are strictly increasing and strictly decreasing on the same intervals.
\item The functions \(\sin(\alpha(t))\) and \(\sin\bigl(\tfrac{2\pi}{T}t\bigr)\) have the same symmetries and the same critical points. Moreover, they are strictly increasing and strictly decreasing on the same intervals.
\end{itemize}
\end{lemma}

\begin{figure}[ht]
  \centering
  \includegraphics[width=0.6\textwidth]{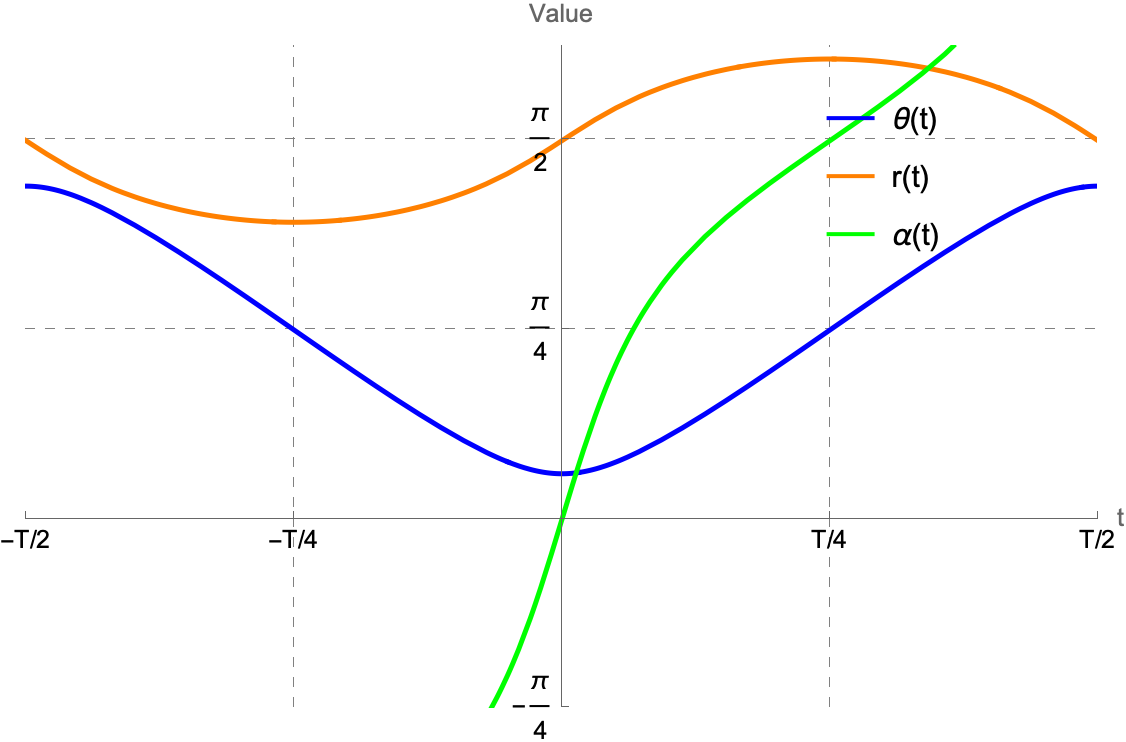}
  \caption{Plots of $\theta(t)$, $r(t)$, and $\alpha(t)$ over one period. This is a translation of the solution described in Theorem~\ref{ThmCS}. Here $r_0=r(-T/4)$ and $\theta_0=\theta(0)$ is the minimum of $\theta$.}
  \label{fig:theta-r-alpha}
\end{figure}

\section{Spectrum of the Laplacian and the Stability Operator}

We begin with the following lemma.

\begin{lemma}\label{formula-laplacian}
Let $\zeta(t,y,z)$ be a smooth function defined along the immersion $X: \bfS{n-1} \times \bfS{n-1} \times \bfS{1} \to \bfS{2n}$. Then
\[
\Delta\zeta
=\zeta_{tt}
+\tfrac{(n-1)}{2}\,\bigl(\ln(EG)\bigr)'\,\zeta_t
+\frac{1}{E}\,\Delta_{\bfS{n-1}}^y \zeta
+\frac{1}{G}\,\Delta_{\bfS{n-1}}^z \zeta,
\]
where 
\[
E = \sin^2(r(t)) \cos^2(\theta(t)), 
\quad 
G = \sin^2(r(t)) \sin^2(\theta(t)),
\]
and $\Delta_{\bfS{n-1}}^y \zeta$ (resp.\ $\Delta_{\bfS{n-1}}^z \zeta$) denotes the Laplacian on the sphere $\bfS{n-1}$ acting on the $y$-variable (resp.\ the $z$-variable), with $t$ and $z$ (resp.\ $t$ and $y$) held fixed.
\end{lemma}

\begin{proof}
The induced metric has the warped product form
\[
g = dt^2 + E(t)\,g_{\bfS{n-1}}^y + G(t)\,g_{\bfS{n-1}}^z,
\]
with $E(t)=\sin^2(r(t))\cos^2(\theta(t))$ and $G(t)=\sin^2(r(t))\sin^2(\theta(t))$, where $g_{\bfS{n-1}}^y$ and $g_{\bfS{n-1}}^z$ are the standard round metrics on the $\bfS{n-1}$ factors. For a metric of this form,
\[
\Delta \zeta 
= \frac{\partial^2 \zeta}{\partial t^2} 
+ (n-1)\left( \frac{E'}{2E} + \frac{G'}{2G} \right)\frac{\partial \zeta}{\partial t}
+ \frac{1}{E}\, \Delta_{\bfS{n-1}}^y \zeta
+ \frac{1}{G}\, \Delta_{\bfS{n-1}}^z \zeta,
\]
Since $\bigl(\ln(EG)\bigr)' = \frac{E'}{E} + \frac{G'}{G}$, this is exactly the stated formula.
\end{proof}

Before establishing the main theorem in this section it is a good idea to recall the following well-known theorem on the eigenvalues of the Laplace operator on spheres

\begin{theorem}\label{SpectrumSphere}  Let $\bfS{k}$ denote the $k$-dimensional unit sphere endowed with the standard metric as a subset of $\bfR{k+1}$. The eigenvalues of the Laplace operator $-\Delta$ are

$$\alpha_1=0,\quad \alpha_2=k, \quad \hbox{in general,} \quad \alpha_i=(i-1)(k+i-2)$$

with multiplicities

$$m_1=1,\quad m_2=k+1, \quad \hbox{and for $i>2$,} \quad m_i=\binom{k+i-1}{i-1}-\binom{k+i-3}{i-3}$$

\end{theorem}

\begin{theorem}\label{operatorsLS}
Consider the stability operator on the minimal embedding $X:\bfS{n-1}\times \bfS{n-1}\times \bfS{1}\to \bfS{2n}$ given by
\[
J=-\Delta-|A|^2-(2n-1),
\]
where \(|A|^2\) is the squared norm of the shape operator. For each pair of positive integers \(i,j\), define on the space of \(T\)-periodic functions \(\eta(t)\):
\[
L_{i,j}[\eta]
= -\eta''(t)
 - \tfrac{n-1}{2}\bigl(\ln(EG)\bigr)'(t)\,\eta'(t)
 + \frac{\alpha_i}{E(t)}\,\eta(t)
 + \frac{\alpha_j}{G(t)}\,\eta(t),
\]
\[
S_{i,j}[\eta]
= L_{i,j}[\eta]
 - \bigl(|A|^{2}(t)+(2n-1)\bigr)\,\eta(t).
\]
Let \(\{\lambda_{i,j}^{(k)}\}_{k\ge1}\) and \(\{\mu_{i,j}^{(k)}\}_{k\ge1}\) be the increasing sequences of eigenvalues of
\[
L_{i,j}[\eta]=\lambda\,\eta,
\qquad
S_{i,j}[\eta]=\mu\,\eta,
\]
respectively. Then
\[
\text{Spectrum}(-\Delta)
=\bigcup_{i,j\ge 1}\{\lambda_{i,j}^{(k)}:k=1,2,\dots\},
\qquad
\text{Spectrum}(J)
=\bigcup_{i,j\ge 1}\{\mu_{i,j}^{(k)}:k=1,2,\dots\}.
\]

Moreover, if $\lambda$ is an eigenvalue of $-\Delta$ and
\[
\lambda=\lambda_{{i_1}, {j_1}}^{k_1}=\dots =\lambda_{{i_p}, {j_p}}^{k_p},
\]
then the multiplicity of $\lambda$ is
\[
n_1\,\bar{m}_{k_1}+\dots+n_p\,\bar{m}_{k_p},
\]
where $n_\ell$ is the multiplicity of $\lambda$ as an eigenvalue of the operator $L_{i_\ell,j_\ell}$ and
\[
 \bar{m}_{k_\ell}= 
 m_{i_\ell}m_{j_\ell}
\]
A similar result holds for the multiplicities of the eigenvalues of the stability operator.
\end{theorem}

\begin{proof}
The proof follows the separation-of-variables argument in \cite{P} and \cite{P2020}. 
Every eigenfunction of \(-\Delta\) can be written as
\[
\zeta(t,y,z)=\eta(t)\,f_i(y)\,f_j(z),
\]
where \(\eta\) is \(T\)-periodic and satisfies
\[
L_{i,j}[\eta]=\lambda\,\eta,
\]
and where \(f_i(y)\) (resp.\ \(f_j(z)\)) is an eigenfunction of the Laplacian on \(\bfS{n-1}\) with eigenvalue \(\alpha_i\) (resp.\ \(\alpha_j\)). Thus the spectrum of \(-\Delta\) is the union of the spectra of the operators \(L_{i,j}\); the same separation-of-variables argument applies to \(J\).
The multiplicity statement follows by counting the independent products \(f_i(y)f_j(z)\) together with the algebraic multiplicity of \(\lambda\) as an eigenvalue of \(L_{i,j}\).
\end{proof}

The following lemma points out how to change our differential equations into Hill's equations.

\begin{lemma}\label{tildas}
A $T$-periodic function $\eta$ is an eigenfunction of the operator $L_{i,j}$ if and only if $(EG)^{\frac{n-1}{4}}\eta$ is an eigenfunction of the operator 
\[
\tilde{L}_{i,j}(z)
=-z''+\left(\frac{n-1}{4}\frac{(EG)''}{EG}-\frac{(n-1)(5-n)}{16}\left(\frac{(EG)'}{EG}\right)^2+\frac{\alpha_i}{E}+\frac{\alpha_j}{G}\right) z.
\]
Likewise, a $T$-periodic function $\eta$ is an eigenfunction of the operator $S_{i,j}$ if and only if $(EG)^{\frac{n-1}{4}}\eta$ is an eigenfunction of the operator
\[
\tilde{S}_{i,j}(z)=\tilde{L}_{i,j}(z)-\bigl(|A|^2+(2n-1)\bigr)z.
\]
Moreover, the functions 
\[
V(t)=\frac{n-1}{4}\frac{(EG)''}{EG}-\frac{(n-1)(5-n)}{16}\left(\frac{(EG)'}{EG}\right)^2+\frac{\alpha_i}{E}+\frac{\alpha_j}{G}
\quad \hbox{and}\quad 
V(t)-\bigl(|A|^2+(2n-1)\bigr)
\]
are even functions.
\end{lemma}

\begin{proof}
The proof is a direct and standard verification. The evenness of $V$ follows from Lemma \ref{simetrias}.
\end{proof}

Let us finish this section with the following theorems on Hill's equation. Their proofs can be found in \cite{MW} and \cite{CH}.

\begin{theorem}\label{discriminant}
Consider the differential equation
\begin{equation}\label{Hill}
z''(t)+(\lambda+Q(t))z(t)=0
\end{equation}
where $Q$ is a smooth $T$-periodic function. For any $\lambda$ let us define
\[
\delta(\lambda):= z_1(T,\lambda)+z_2'(T,\lambda),
\]
where $z_1(t,\lambda)$ and $z_2(t,\lambda)$ are solutions of \eqref{Hill} such that $z_1(0,\lambda)=1$, $z_1'(0,\lambda)=0$ and $z_2(0,\lambda)=0$, $z_2'(0,\lambda)=1$. There exists an increasing infinite sequence of real numbers $\lambda_1,\lambda_2,\ldots$ such that the differential equation \eqref{Hill} has a $T$-periodic solution if and only if $\lambda=\lambda_j$. Moreover, the $\lambda_j$ are the roots of the equation $\delta(\lambda)=2$. The function $\delta$ is called the discriminant function of the operator $K[z]=z''(t)+Q(t)z(t)$.
\end{theorem}

\begin{theorem}\label{nodal}
Let us denote by $\lambda_1<\lambda_2\le \lambda_3\le \lambda_4\cdots$ the sequence of eigenvalues of the Hill’s equation presented in Theorem \ref{discriminant}. If $z(t)$ is a nonzero $T$-periodic solution of \eqref{Hill} with $\lambda=\lambda_i$, then the number of zeros of $z(t)$ in the interval $[0,T)$ is $2\bigl\lfloor \tfrac{i}{2}\bigr\rfloor$.
\end{theorem}

\begin{theorem}\label{even}
Let $Q(t)$, $z_1(t)$ and $z_2(t)$ be as in Theorem \ref{discriminant}. If $Q(t)$ is an even function, then:
\begin{enumerate}
\item $z_1(T)=2 z_1(T/2) z_2'(T/2)-1=1+2 z_1'(T/2)z_2(T/2)$,
\item $z_2(T)=2 z_2(T/2)z_2'(T/2)$,
\item $z_1'(T)=2 z_1(T/2)z_1'(T/2)$,
\item $z_2'(T)=z_1(T)$.
\end{enumerate}
Moreover, we have that for any $\lambda_k$, the discriminant function $\delta(\lambda)$ satisfies that 

$$\delta'(\lambda_k)=z_1'(T)\int_0^Tz_2(t)^2\, dt \, .$$

In all cases, $z_1(t)$ is an even function and $z_2(t)$ is an odd function. Whenever a nontrivial $T$-periodic solution of \eqref{Hill} exists, there also exists one which is either odd or even. Therefore, these periodic solutions are necessarily multiples of one of the normalized solutions $z_1(t)$ or $z_2(t)$ unless all solutions are periodic.
\end{theorem}

\begin{theorem}\label{derivative-discriminant}
If $\delta(\lambda)$ is the discriminant function defined in Theorem \ref{discriminant}, then $\delta'(\lambda_1)<0$.
\end{theorem}

\begin{theorem}[Rayleigh characterization]\label{thm:rayleigh}
Let $Q$ be a smooth $T$-periodic function and let $\{\lambda_k\}$ be as in Theorem \ref{discriminant}. 
For any nonzero $T$-periodic function $z$, define the Rayleigh quotient
\[
\mathcal{R}[z]=
-\frac{\displaystyle\int_0^T z(t)\,(z''(t)+Q(t)\,z(t))\,dt}
     {\displaystyle\int_0^T z(t)^2\,dt}.
     =
     \frac{\displaystyle\int_0^T z'(t)^2-Q(t)\,z(t)^2\,dt}
     {\displaystyle\int_0^T z(t)^2\,dt}.
\]
If $W$ is a $k$-dimensional subspace of $T$-periodic functions such that 
\[
\sup_{0\neq z\in W}\,\mathcal{R}[z]\;\le\;M_W,
\]
then 
\[
\lambda_k\;\le\;M_W.
\]
\end{theorem}

\section{Bounds on the stability index}

In this section we are heavily using the fact that our solution $r(t), \theta(t)$ and $\alpha(t)$ have no critical points between $0$ and $T/4$ and that $\theta(0)=\theta_0$ is the minimum of the function $\theta$ and that $r_0=r(-T/4)=r(3T/4)$ is the minimum of the function $r(t)$. See Figure \ref{fig:theta-r-alpha} and Lemma \ref{simetrias}.

\begin{lemma} \label{evenodd} Let $\gamma_1,\gamma_2,\gamma_3,\nu_1,\nu_2,\nu_3$ be the functions defined in Equations \eqref{eq:ImmersionComponents} and \eqref{eq:GaussComponents}. We have:
\begin{enumerate}
\item
 $\gamma_1$ and $\gamma_2$ are positive even functions.
  \item
  $\gamma_3$ is odd and vanishes only at $t=T/2$ (on $(0,T)$).
 \item
 $\nu_1$, $\nu_2$ and $\nu_1\gamma_2-\nu_2\gamma_1$  are even functions that do not vanish at $t=0$ and have at least two zeros in the interval $[0,T)$.
 \item
 The functions $\nu_3$ and  $\nu_1\gamma_3-\nu_3\gamma_1$ and $\nu_2\gamma_3-\nu_3\gamma_2$ are odd and vanish at $T/2$
\end{enumerate}
\end{lemma}

\begin{proof} Parts (1) and (2) follow from Lemma \ref{simetrias} and the simple expressions for $\gamma_1,\gamma_2$ and $\gamma_3$.
To verify the rest of the proposition it is convenient to notice that,
\begin{eqnarray*}
\nu_1\gamma_2-\nu_2\gamma_1&=&\sin r \cos \alpha \\
\nu_1\gamma_3-\nu_3\gamma_1&=&\cos r \cos \alpha  \sin \theta +\sin \alpha \cos \theta\\
\nu_2\gamma_3-\nu_3\gamma_2&=&\sin \alpha \sin \theta -\cos r \cos \alpha  \cos \theta\\
\end{eqnarray*}
From Lemma \ref{simetrias} we have that the functions $\nu_1$, $\nu_2$ and $\nu_1\gamma_2-\nu_2\gamma_1$ are even, symmetric with respect to the vertical line $t=T/2$. Since $\nu_1(0)=\sin(\theta_0)>0$ and $\nu_1(T/2)=-\cos(\theta_0)<0$ then $\nu_1$ has  one zero on $[0,T/2]$ and at least two in $[0,T]$. Likewise for $\nu_2$ because $\nu_2(0)=-\cos(\theta_0)<0$ and $\nu_2(T/2)=\sin(\theta_0)>0$, and for $\nu_1\gamma_2-\nu_2\gamma_1$ because $\nu_1\gamma_2-\nu_2\gamma_1$ is 1 at $t=0$ and $-1$ at $t=T/2$. By Lemma \ref{simetrias} we have that the functions $\nu_1\gamma_3-\nu_3\gamma_1$ and $\nu_2\gamma_3-\nu_3\gamma_2$ are odd. We can directly check  that these two functions vanish at $t=T/2$.
\end{proof}

Now we point out some natural solutions for the operators $S_{ij}$ and $L_{ij}$.

\begin{lemma} \label{knownsolutions} Let $\gamma_1,\gamma_2,\gamma_3,\nu_1,\nu_2,\nu_3$ be the functions defined in Equations \eqref{eq:ImmersionComponents} and \eqref{eq:GaussComponents} and $S_{ij}$ be the operators defined in Theorem \ref{operatorsLS}. We have:
\begin{enumerate}
\item $$S_{21}(\nu_1)=-(2n-1)\,\nu_1,\quad S_{12}(\nu_2)=-(2n-1)\,\nu_2\, \quad S_{11}(\nu_3)=-(2n-1)\,\nu_3$$

\item $$L_{21}(\gamma_1)=(2n-1)\,\gamma_1,\quad L_{12}(\gamma_2)=(2n-1)\,\gamma_2\, \quad L_{11}(\gamma_3)=(2n-1)\,\gamma_3$$
  \item
If for $1\le i<j\le 3$ we define $f_{ij}=\nu_i\gamma_j-\nu_j\gamma_i$, then

$$ S_{22}(f_{12})=0,\quad S_{21}(f_{13})=0 \, \quad S_{12}(f_{23})=0$$
\end{enumerate}
\end{lemma}

\begin{proof}
We can directly check these differential equations or we can use Lemma \ref{formula-laplacian} and the fact that the immersion is given by 

$$X=\left(\gamma_1\, y_1,\dots \gamma_1\, y_n,\gamma_2\, z_1,\dots, \gamma_2\, z_n,\gamma_3\right)$$
and the Gauss map is given by 
$$\nu=\left(\nu_1\, y_1,\dots \nu_1\, y_n,\nu_2\, z_1,\dots, \nu_2\, z_n,\nu_3\right)$$

where the function $y_i$ and $z_i$ are the coordinates of $S^{n-1}$. Recall that it is well known that

$$-\Delta_{\bfS{n-1}}^y (y_i)=\alpha_2\, y_i=(n-1)\, y_i\quad\hbox{and}\quad -\Delta_{\bfS{n-1}}^z (z_i)=\alpha_2\, z_i=(n-1)\, z_i$$

and from Equation \eqref{delta-J} we have that 

$$J(\nu_1\, y_1)=-(2n-1) \nu_1\, y_1,\quad J(\nu_2\, z_1)=-(2n-1) \nu_2\, z_1 \quad\hbox{and}\quad J(\nu_3)=-(2n-1) \nu_3$$

and  

$$J(f_{12}\, y_1 z_1)=0,  \quad J(f_{13}\, y_1)=0 \quad\hbox{and}\quad  J(f_{23}z_1)=0$$

and

$$-\Delta(\gamma_1\, y_1)=(2n-1) \gamma_1\, y_1,\quad -\Delta(\gamma_2\, z_1)=(2n-1) \gamma_2\, z_1 \quad\hbox{and}\quad -\Delta(\gamma_3)=(2n-1) \gamma_3$$

\end{proof}

\begin{theorem}\label{thm1}
The stability index of the minimal immersion
\[
X(y,z,t) = \bigl( \sin r(t)\,\cos\theta(t)\,y,\; \sin r(t)\,\sin\theta(t)\,z,\; \cos r(t) \bigr)
\]
is at least $n^2+4n+3$.
\end{theorem}
\begin{proof}
We use Lemmas~\ref{knownsolutions}, \ref{evenodd}, \ref{tildas} and Theorem~\ref{nodal}.

From Lemmas~\ref{knownsolutions} and \ref{tildas},
\[
\tilde{S}_{11}\!\left((EG)^{\frac{n-1}{4}}\,\nu_3\right)
= -\bigl(2n-1\bigr)\,(EG)^{\frac{n-1}{4}}\,\nu_3 .
\]
By Lemma~\ref{evenodd}, $(EG)^{\frac{n-1}{4}}\nu_3$ is odd and has at least two zeros in $[0,T)$.  
Theorem~\ref{nodal} then implies that $-(2n-1)$ is not the first eigenvalue of $\tilde{S}_{11}$.

Similarly, $h_1=(EG)^{\frac{n-1}{4}}\nu_1$ has at least two zeros in $[0,T)$ and satisfies
$\tilde{S}_{21}(h_1) = -(2n-1)h_1$, hence $S_{21}$ has at least two negative eigenvalues.  
The same argument for $h_2=(EG)^{\frac{n-1}{4}}\nu_2$ shows that $S_{12}$ also has at least two negative eigenvalues.

For $h_3 = (EG)^{\frac{n-1}{4}} f_{12}$, we have at least two zeros in $[0,T)$ and
$\tilde{S}_{22}(h_3) = 0\cdot h_3$, so $0$ is not the first eigenvalue; therefore $S_{22}$ has at least one negative eigenvalue.

Since $S_{21}=S_{11}+\dfrac{\alpha_2}{E}$ with $\dfrac{\alpha_2}{E}>0$ pointwise, the min–max (Rayleigh) principle, Theorem \ref{thm:rayleigh}, gives
\[
\mu^{21}_k \;>\; \mu^{11}_k \qquad \text{for all } k.
\]
If $-(2n-1)$ were the second eigenvalue of $S_{11}$, i.e. $\mu^{11}_2=-(2n-1)$, then
\[
\mu^{21}_2 \;>\; \mu^{11}_2 \;=\; -(2n-1).
\]
But $-(2n-1)$ is an eigenvalue of $S_{21}$ whose eigenfunction has at least two zeros, hence by Theorem~\ref{nodal} it cannot be the first eigenvalue, which contradicts $\mu^{21}_2>-(2n-1)$. Therefore $S_{11}$ has at least three negative eigenvalues.

By Theorem~\ref{operatorsLS}, the multiplicities of the negative eigenvalues contribute as follows:  
$3$ from $S_{11}$, $2n$ from $S_{21}$, $2n$ from $S_{12}$, and $n^2$ from $S_{22}$.  
Hence the stability index is at least
\[
3 + 2n + 2n + n^2 = n^2+4n + 3. \qedhere
\]
\end{proof}



\section{On the first nonzero eigenvalue of the Laplacian}

\begin{theorem}
Let $z_1(t)$ be the solution of the initial value problem
\begin{equation}\label{ode1}
z_1''(t)+\frac{n-1}{2}\bigl(\ln(EG)\bigr)'(t)\,z_1'(t)+(2n-1)\,z_1(t)=0,
\qquad
z_1(0)=1,\ \ z_1'(0)=0.
\end{equation}
Then, the first nonzero eigenvalue of $-\Delta$ for the minimal embedded hypersurface in $S^{2n}$
\[
X(y,z,t) = \bigl( \sin r(t)\,\cos\theta(t)\,y,\; \sin r(t)\,\sin\theta(t)\,z,\; \cos r(t) \bigr)
\]
is $2n-1$ if and only if $z_1'(T)>0$.
\end{theorem}

\begin{proof}
Since the functions $\gamma_1$ and $\gamma_2$ are positive (see Lemma \ref{evenodd}) and
\[
\tilde{L}_{21}\bigl((EG)^{\frac{n-1}{4}}\gamma_1\bigr)=(2n-1)\,(EG)^{\frac{n-1}{4}}\gamma_1,
\qquad
\tilde{L}_{12}\bigl((EG)^{\frac{n-1}{4}}\gamma_2\bigr)=(2n-1)\,(EG)^{\frac{n-1}{4}}\gamma_2
\]
(Lemma \ref{knownsolutions}), it follows that $(2n-1)$ is the first eigenvalue of $\tilde{L}_{21}$ and of $\tilde{L}_{12}$.
As in Theorem \ref{thm1}, the min–max (Rayleigh) principle implies that if either $i>1$ or $j>1$, then the first eigenvalue of $\tilde{L}_{ij}$ is greater than or equal to $2n-1$.

By Theorem~\ref{operatorsLS}, $2n-1$ is the first nonzero eigenvalue of $-\Delta$ if and only if $2n-1$ is the first nonzero eigenvalue of $\tilde{L}_{11}$.

Since $L_{11}(1)=0$, we have that $0$ is the first eigenvalue of both $L_{11}$ and $\tilde{L}_{11}$. Moreover, $\gamma_3$ vanishes exactly two times on $[0,T)$, hence by Theorem \ref{nodal} the value $2n-1$ is either the second or the third eigenvalue of $\tilde{L}_{11}$.

Consider now the discriminant for the Hill equation
\[
z''(t)+\bigl(\lambda - V(t)\bigr) z(t)=0,
\]
where
\[
V(t)=\frac{n-1}{4}\frac{(EG)''}{EG}-\frac{(n-1)(5-n)}{16}\left(\frac{(EG)'}{EG}\right)^{\!2}.
\]
Let $\{\lambda_k\}$ be the sequence defined in Theorem \ref{discriminant}. Then $\lambda_1=0$, and either $\lambda_2=2n-1$ or $\lambda_3=2n-1$.

By Lemma \ref{tildas}, $V$ is an even function. By Theorem \ref{derivative-discriminant}, we have $\delta'(0)<0$ and $\delta(0)=2$. Hence $\delta(\lambda)<2$ for all $\lambda\in(0,\lambda_2)$. Therefore $\lambda_2=2n-1$ if and only if $\delta'(2n-1)>0$.

By Theorem \ref{even} (the even–potential identities for the discriminant), at $\lambda=2n-1$ we have
\[
\delta'(2n-1)=\tilde{z}_1'(T)\int_0^T \tilde{z}_2(t)^2\,dt,
\]
where $\tilde{z}_1,\tilde{z}_2$ are the normalized fundamental solutions of $z''+(\lambda-V)z=0$. By Theorem \ref{tildas}, $\tilde{z}_1=(EG)^{\frac{n-1}{4}}\,z_1$, where $z_1$ solves \eqref{ode1}. Since $(EG)^{\frac{n-1}{4}}$ is even and $T$–periodic, its derivative vanishes at $t=T$, hence $\tilde{z}_1'(T)$ and $z_1'(T)$ have the same sign. Consequently,
\[
\delta'(2n-1)>0 \quad\Longleftrightarrow\quad z_1'(T)>0.
\]
This is equivalent to $2n-1$ being the first nonzero eigenvalue of $\tilde{L}_{11}$, and therefore of $-\Delta$, as claimed.
\end{proof}

\end{document}